\documentclass[11pt, reqno, english]{amsart}
\usepackage{accents}
\usepackage{style}
\usepackage[normalem]{ulem}
\usepackage{wasysym}
\usepackage{comment}

\begin{document}

\title{Quantifying discontinuity}

\author[Adams]{Henry Adams}
\address[HA]{Department of Mathematics, University of Florida, Gainesville, FL 32611, USA}
\email{henry.adams@ufl.edu}

\author[Frick]{Florian Frick}
\address[FF]{Dept.\ Math.\ Sciences, Carnegie Mellon University, Pittsburgh, PA 15213, USA}
\email{frick@cmu.edu} 

\author[Harrison]{Michael Harrison}
\address[MH]{Microsoft Research, Seattle, USA}
\email{mah5044@gmail.com}

%\author[Lagoda]{Evgeniya Lagoda}
%\address[EL]{Freie Universit\"at Berlin}
%\email{e.lagoda@fu-berlin.de}

%\author[Lim]{Sunhyuk Lim}
%\address[SL]{Department of Mathematics, Sungkyunkwan University (SKKU), Suwon, South Korea}
%\email{lsh3109@skku.edu}

\author[Sadovek]{Nikola Sadovek}
\address[NS]{Dept.\ Math.\ Sciences, Carnegie Mellon University, Pittsburgh, PA 15213, USA}
\email{nsadovek@andrew.cmu.edu}

\author[Superdock]{Matt Superdock}
\address[MS]{Department of Computer Science, Rhodes College, Memphis, TN 38112, USA}
\email{superdockm@rhodes.edu}

\thanks{\hspace{-4mm}\textit{2020 Mathematics Subject Classification.} 51F30, 57N35, 52A35. \\
HA was supported by the NSF CAREER Grant DMS 2540172 and by the Simons Foundation Travel Support for Mathematicians.
FF was supported by NSF CAREER Grant DMS 2042428.
NS funded in part by the Deutsche Forschungsgemeinschaft (DFG, German Research Foundation) under Germany's Excellence Strategy--The Berlin Mathematics Research Center MATH+ (EXC-2046/1, project ID 390685689, BMS Stipend).}

\begin{abstract}
\small 
{Given a compact space $X$ that does not admit an embedding (an injective continuous function) into $\R^d$, we study the discontinuity that any injective function $X \to \R^d$ must have.
To this end, we define a scale invariant modulus of discontinuity and obtain general lower bounds, thus obtaining quantified nonembeddability results of Haefliger--Weber type. Moreover, we establish analogous lower bounds for simplicial complexes that do not admit an almost $r$-embedding in $\R^d$, thus obtaining a quantified version of the topological Tverberg theorem.}
\end{abstract}

\maketitle

%
%DELETE T.O.C. BEFORE FINAL VERSION
%

%\setcounter{tocdepth}{1}
%\tableofcontents

%=============

\section{Introduction}
\label{sec:intro}

The problem of determining whether a space $X$ embeds in~$\R^d$ is a classical one in topology.
Seminal results include constructions of Flores~\cite{flores1933n} and of van Kampen \cite{van1933komplexe} in the 1930s of $d$-dimensional complexes that do not embed into~$\R^{2d}$ (on the negative side), and Whitney's embedding theorem~\cite{whitney44} from the 1940s which asserts that any smooth $d$-manifold embeds into~$\R^{2d}$ (on the positive side).
If $X$ indeed embeds into~$\R^d$ we may ask how rich the space of embeddings $X \hookrightarrow \R^d$ is---for example, we can investigate whether all such embeddings are the same up to isotopy (the ``unknotting problem'').
If $X$ is endowed with a metric, quantitative bounds for embeddings, such as for their Lipschitz constants, have been studied; see Gromov and Guth~\cite{gromov2012generalizations} and Freedman and Krushkal~\cite{freedman2014geometric}.

Here we complement this quantitative investigation on the negative side, that is, for nonembedd\-ability results.
We augment classical results asserting that $X$ does not embed into~$\R^d$ to quantified bounds for the discontinuity of injective functions $X \to \R^d$.
For compact $X$ a nonembeddability result simply asserts that there is no continuous injection of $X$ into~$\R^d$.
We more generally establish lower bounds for a measure of discontinuity of injective functions $X \to \R^d$.
This measure of discontinuity needs to be scale-invariant, since for a given injective function $f\colon X\to\R^d$, a suitable rescaled function has its image contained in some small $\varepsilon$-ball.
In particular, the modulus of discontinuity (also called oscillation) of injective functions $X\to \R^d$ is arbitrarily small.
%It neither suffices to restrict attention to those injective functions $X\to \R^d$ whose image has diameter one, since starting with an injective function whose image is contained in some small $\varepsilon$-ball, we can in a continuous way stretch out a topologically uninteresting part of~$X$ until the image attains a desired diameter.

Instead, a suitable measure of discontinuity will only make reference to angles and not to distances in~$\R^d$.
Let $f\colon X\to \R^d$ be continuous.
For $u,v \in \R^d$ denote the line segment connecting $u$ to~$v$ by~$\overline{uv}$.
For any given angle $\alpha >0$ and any point $(x,y) \in X\times X$ with $x \ne y$, there is a neighborhood $U$ of~$(x,y)$ such that for all $(x',y') \in U$ the line segments $\overline{f(x)f(y)}$ and $\overline{f(x')f(y')}$ make an angle of at most~$\alpha$.
For discontinuous~$f$, even if $x$ is close to~$x'$ and $y$ is close to~$y'$, the corresponding line segments connecting their images in~$\R^d$ could make a large angle.
We will refer to the infimal~$\alpha$ such that for any point $(x,y) \in X\times X$, there is a neighborhood $U$ of~$(x,y)$ such that for all $(x',y') \in U$ the line segments $\overline{f(x)f(y)}$ and $\overline{f(x')f(y')}$ make an angle of at most~$\alpha$ as the \emph{scale-invariant modulus of discontinuity}~$\alpha(f)$.
In particular, continuous functions $f$ necessarily have $\alpha(f) = 0$.
Moreover, we show in Section \ref{sec:suitability-alpha} that under mild assumptions $\alpha$ completely captures continuity in the sense that $\alpha(f) = 0$ happens if and only if $f$ is continuous.
Therefore, bounding $\alpha(f)$ away from zero indeed represents a scale-invariant quantification of discontinuity of~$f$.

Here we list some consequences of our results for classical nonembeddability theorems, starting with the quantified version of the nonembeddability of $\RP^d$ into $\R^{2d-1}$ when $d$ is a power of two~\cite{milnor1974characteristic}.
We will denote by $r_n$ the diameter of the vertices of a regular $(n+1)$-simplex inscribed in $S^n$ (see Section~\ref{sec:background}).

\begin{theorem} \label{thm:RP}
    Any injective function $f \colon \RP^{2^k} \to \R^{2^{k+1}-1}$ satisfies $\alpha(f) \ge r_{2^{k+1}-2} = \arccos{\left(\frac{-1}{2^{k+1}-1}\right)}$.
\end{theorem}

This is a consequence of the more general Theorem \ref{thm:quant-general} and Corollary \ref{cor:quant-hw} for manifolds and simplicial complexes.
This approach relies on the seminal work of Haefliger~\cite{Haefliger} and Weber~\cite{Weber} relating the embedding problem to the existence of $\Z/2$-equivariant maps $\Conf_2(X) \to S^{d-1}$  (see Theorem \ref{thm:Haefliger}).
For spaces with well-enough understood configuration space one can go beyond the Haefliger--Weber metastable range to obtain an improved lower bound:

\begin{theorem} \label{thm:StoR}
    For $k \ge d-1$, any injective function $f \colon S^{k} \to \R^d$ satisfies $\alpha(f) \ge c_{d - 1, k}$.
\end{theorem}

This is a special case of a more general Theorem~\ref{thm:alpha-coind}; the constants~$c_{d - 1, k}$ are defined in Section~\ref{sec:background} and record certain homotopy information of the Vietoris--Rips complexes of spheres.

One of the earliest nonembeddability results asserts the non-planarity of the complete graph on five vertices~$K_5$.
Our results more generally show that for any injective function $f\colon K_5\to \R^2$ there are tuples of arbitrarily close points $(x,y)$ and $(x',y')$ in~$K_5$ such that the line segments $\overline{f(x)f(y)}$ and $\overline{f(x')f(y')}$ make an angle of at least~$\arccos(-1/2) = 2\pi/3$.
Namely, we obtain the following quantified version of the nonembeddability result of van Kampen and Flores:

\begin{theorem} \label{thm:vKF-intro}
    Every injective function $\sk_d(\Delta_{2d+2}) \to \R^{2d}$ satisfies $\alpha(f) \ge r_{2d-1} = \arccos{\left(\frac{-1}{2d}\right)}$.
\end{theorem}

In Section~\ref{subsec:vk-f} we study a wider notion of an \emph{almost injective} function from a simplicial complex $K$ into~$\R^d$, where the function does not identify two points from disjoint faces of~$K$.
In Theorem~\ref{thm:quant-vk-f} we prove a quantified extension of the classical van Kampen--Flores theorem which says that there is no almost injective function $\sk_d(\Delta_{2d+2}) \to \R^{2d}$.
The formulation of this result uses a slight variation $\alpha^{(2)}$ of the scale invariant modulus of discontinuity $\alpha$ adapted to almost embeddings.

Furthermore, Tverberg-type theory studies \emph{almost $r$-injective} functions~$f \colon~K \to \R^d$ for an integer $r \ge 2$, where $f$ does not identify $r$ points of $K$ lying in pairwise disjoint faces.
A continuous almost $r$-injective function is in the literature also called an \emph{almost $r$-embedding} \cite{mabillard2014eliminating}.
The topological Tverberg conjecture states that there is no almost $r$-embedding from the $(r-1)(d+1)$-dimensional simplex $\Delta_{(r-1)(d+1)}$ into $\R^d$.
B\'ar\'any, Shlosman and Sz\H ucs~\cite{BSSz1981} and \"Ozaydin~\cite{ozaydin1987} proved the conjecture for $r$ a prime power, while Frick~\cite{frick2015counterexamples} (see also Blagojevi\'c, Frick, and Ziegler~\cite{blagojevic2015barycenters}) established counterexamples for composite $r \le d/3$ using the $r$-fold Whitney trick of Mabillard and Wagner~\cite{mabillard2014eliminating}.
This was later improved upon by Avvakumov, Mabillard, Skopenkov, and Wagner~\cite{avvakumov2021eliminating} who obtained counterexamples for $r \le (d-1)/2$.
To this end, in Section \ref{sec:almost-r-emb} we introduce the $r$-fold version $\alpha^{(r)}$ of the scale invariant modulus of discontinuity to study and quantify the failure of the existence of almost $r$-injective functions from the simplex.
Namely, we establish the following:

\begin{theorem} [Quantified topological Tverberg] \label{thm:tv-intro}
    Let $d \ge 1$ be an integer and $r \ge 2$ a prime power.
Then, any almost $r$-injective function $f\colon \Delta_{(r-1)(d+1)} \to \R^d$ satisfies $\alpha^{(r)}(f) \ge  \arccos(\tfrac{-1}{d(r-1)})$.
\end{theorem}

As before, this theorem represents a quantified extension of the topological Tverberg theorem, as continuous almost $r$-injective functions (i.e., almost $r$-embeddings) necessarily have $\alpha^{(r)}=0$.
We also obtain the following generalization of the topological Tverberg theorem as a consequence, where the quantity
\begin{equation*}
    \kappa^{(r)}(f) = \inf\, \frac{1}{r} \Bigg(\sum_{i = 1}^{r} \sum_{j = 1}^{r} \|f(x_{i}) - f(x_{j})\|^{2}\Bigg)^{1/2}
\end{equation*}
represents the $r$-fold version of the Gromov--Guth thickness~\cite{gromov2012generalizations} of a function $f\colon \Delta_{(r - 1)(d + 1)} \to \R^{d}$.
Here the infimum is taken over all $r$-tuples of points which lie in pairwise disjoint faces.

\begin{theorem}
\label{thm:tv-thick-intro}
    Let $d \ge 1$ be an integer, $r$ a prime power, and $f \colon \Delta_{(r - 1)(d + 1)} \to \R^{d}$ any function.
    Then $f$ must satisfy $\delta(f) \ge \kappa^{(r)}(f) \cdot \sqrt{1 + \frac{1}{d(r - 1)}}$.
\end{theorem}

The rest of the document is organized as follows.
Section~\ref{sec:background} contains the notation and background.
In Section~\ref{sec:injectivity} we prove quantified nonembeddability results, which represent generalizations of Theorems~\ref{thm:RP} and~\ref{thm:StoR}, as well as the proof of Theorem~\ref{thm:vKF-intro}.
Finally, we develop quantified Tverberg-type theory in Section~\ref{sec:almost-r-emb}, where we prove Theorems~\ref{thm:tv-intro} and~\ref{thm:tv-thick-intro}.

%=============

\section{Background and preliminaries}
\label{sec:background}

%=============

\subsection*{General}

For topological spaces $X, Y$, a continuous function $f\colon~ X \to Y$ will be called a \emph{map} and the term \emph{function} will be used in a general (not necessarily continuous) sense.
We will refer to a topological space equipped with an action by a group $G$ as a \emph{$G$-space}, and we will refer to a function (or a map) $f \colon~ X \to Y$ between $G$-spaces which satisfies $f(g\cdot x)=g\cdot f(x)$, for all $g \in G$ and all $x\in X$, by a \emph{$G$-function} (or a \emph{$G$-map}).
In the case when $G =\Z/2$, we will say that a $\Z/2$-function (or a $\Z/2$-map) is \emph{odd}.

%=============

\subsection*{Discontinuity}

Let $X$ be a topological space, let $Y$ be a metric space, and let $f\colon X \to Y$.
We will use the term \emph{modulus of discontinuity of $f$} from \cite{dubins1981equidiscontinuity} to quantify the extent to which a function between $X$ and $Y$ is discontinuous.
Namely, for $f\colon X \to Y$, it is defined as
\[
    \delta(f) = \inf\{\delta\geq 0~|~\forall x\in X,\ \exists~\text{an open neighborhood } U_x \text{ of } x \text{ such that\ } \diam(f(U_x))\leq \delta\}.
\]
In particular, $f$ is continuous if and only if $\delta(f)=0$.
Dubins and Schwarz's 1981 paper~\cite{dubins1981equidiscontinuity} provides lower bounds on the modulus of discontinuity of odd functions between spheres.
Indeed, equip the $n$-dimensional sphere $S^n$ with the geodesic (path-length) metric, and let
\[
    r_n = \arccos\left(\frac{-1}{n+1}\right)
\]
be the diameter of the vertices of a regular $(n+1)$-simplex inscribed in $S^n$.
Dubins and Schwarz prove that an odd function $f\colon S^{n+1} \to S^n$ (which is necessarily discontinuous by the Borsuk--Ulam theorem) has modulus of discontinuity at least $r_n$, and this constant is tight.

%=============

\subsection*{Vietoris--Rips complexes and metric thickenings}

For a metric space $X$ and a real parameter $r\ge 0$, we denote by $\vr{X}{r}$ the \emph{Vietoris--Rips complex}, which is the simplicial complex whose vertices are the points of $X$, and whose faces are the finite subsets $\sigma \subseteq X$ with the diameter bound $\diam(\sigma) \le r$.
Here we used the ``less than or equal'' convention and there is an alternative ``less than'' convention defined as a subcomplex $\vrless{X}{r} \subseteq \vr{X}{r}$ containing faces of diameter strictly less than $r$.

For a family $\cF$ of subsets of $X$, we will denote by $\vrm{X}{\cF}$ the metric space of all probability measures on $X$ whose support is an element of $\cF$ endowed with a Wasserstein metric.
In the particular case when $\cF$ is the family of all finite subsets of $X$ with diameter at most $r > 0$, we will write $\vrm{X}{r}$ for this space and call it the \emph{Vietoris--Rips metric thickening} of $X$, as introduced in \cite{AAF}.
(We point out that there the space was denoted by $\text{VR}^{\text{m}}(X,r)$.)
Unlike in the case of the Vietoris--Rips complex, the standard inclusion $X \hookrightarrow \vrm{X}{r}$ is a continuous map and, moreover, an isometric embedding.
Similarly as before, there is the ``less than'' analogue $\vrmless{X}{r} \subseteq \vrm{X}{r}$ defined as the subspace of measures whose finite support has a diameter strictly less than $r$.
Throughout the paper we will mostly use $\vrm{X}{r}$, as opposed to $\vrmless{X}{r}$.
%, so we avoid the notation $\vrmleq{X}{r}$ for the former space.

Vietoris--Rips complexes (i.e., their geometric realization) and metric thickenings are closely related.
They have identical underlying sets, but different topologies.
In particular, the identity map induces a \emph{continuous} function $\vr{X}{r} \to \vrm{X}{r}$ and Gillespie \cite{gillespie2024vietoris} proved that
\begin{equation} \label{eq:vr->vrm-weak-equiv}
    \vrless{X}{r} \longrightarrow \vrmless{X}{r}
\end{equation}
is a weak homotopy equivalence.

The homotopy types of $\vrm{S^n}{r}$ are not known in general.
However, Moy \cite{moyVRmS1} showed in the case $n=1$ and $r<\pi$ that $\vrm{S^1}{r} \simeq S^{2k+1}$, where $k \ge 0$ is the unique integer such that $\tfrac{2\pi k}{2k+1}\le r<\tfrac{2\pi(k+1)}{2k+3}$; see also \cite{AA-VRS1}.
For general $n \ge 1$, Adamaszek, Adams, and Frick \cite{AAF} showed
\begin{equation} \label{eq:P(Sn)-homotopy}
    \vrm{S^n}{r} \simeq \begin{cases} S^n, & r < r_{n}\\
    S^n * \frac{\so(n+1)}{A_{n+2}}, & r=r_n,
    \end{cases}
\end{equation}
where $A_{n+2} \subseteq \so(n+1)$ is the group of rotational symmetries of $\Delta_{n+1}$.
In particular, there are $\Z/2$-maps $\vrm{S^n}{r} \to S^n$, for $r < r_n$, and $\vrm{S^n}{r_n} \to S^n * (\so(n+1)/A_{n+2})$.
Here, we assume the $\Z/2$-action on $S^n * (\so(n+1)/A_{n+2})$ to be induced by an antipodal actions on $S^n$ and on the moduli space $\so(n+1)/A_{n+2}$ of regular $(n+1)$-simplices inscribed in $S^n$.

%=============

\subsection*{Constants}

For integers $k \ge n \ge 0$ we define a constant
\begin{equation*}
    c_{n,k} = \inf\{r \ge 0~\colon~\text{$\exists$ an odd map $S^{k} \longrightarrow \vr{S^{n}}{r}$}\}.
\end{equation*}
The known values and bounds on the constants $c_{n, k}$ follow from Theorems~5.1--5.3 of \cite{GH-BU-VR}:
\begin{itemize}
\item $c_{n, n} = 0$,
\item $c_{n, n + 2} = c_{n, n + 1} = r_{n} = \arccos\left(\frac{-1}{n+1}\right)$,
\item $c_{1,2k+1}=c_{1,2k}=\frac{2\pi k}{2k+1}$, and
\item $c_{n,k} \ge \pi - 2\,\cov_{\RP^n}(k)$ for all $k\ge n$; see~\cite{ABF2}.
\end{itemize}
Here $\cov_{\RP^n}(k)$ is the infimum over all $\varepsilon>0$ such that there exists a finite set $A \subseteq \RP^n$ of cardinality $|A| \le k$ so that the balls of radius $\varepsilon$ about $A$ cover $\RP^n$.
We have $c_{n,k}\ge c_{n,k'}$ for all $k\ge k'$, and so the second bullet implies that $c_{n, k} > 0$ for $k > n$.

\begin{comment}
The result of Dubins and Schwarz \cite{dubins1981equidiscontinuity} mentioned above was generalized in the paper~\cite{GH-BU-VR} by the present authors and collaborators: 

\begin{theorem}[Theorems~1.3 and~7.6 of~\cite{GH-BU-VR}]
\label{thm:oddmod}
Any odd function $f\colon S^k \rightarrow S^n$ with $k\geq n$ has modulus of discontinuity $\delta(f)\geq c_{n, k}$, and this bound is tight.
\end{theorem}
\end{comment}

The result of Dubins and Schwarz \cite{dubins1981equidiscontinuity} mentioned above was generalized by the present authors and collaborators in \cite[Theorems~1.3~and~7.6]{GH-BU-VR}, where it is proved that for $k \geq n$, any odd function $f\colon S^k \rightarrow S^n$ has modulus of discontinuity 
\begin{equation} \label{eq:oddmod}
	\delta(f)\geq c_{n, k}
\end{equation}
and that this bound is tight.
We remark that Lim, M\'{e}moli, and Smith~\cite{lim2021gromov} used Dubins and Schwarz' work on the modulus of discontinuity~\cite{dubins1981equidiscontinuity} in order to lower bound the Gromov--Hausdorff distance between spheres of different dimensions.
The existence of odd maps $S^k \to \vr{S^n}{r}$ was studied in Adams, Bush, and Frick~\cite{ABF2}, and related to coverings and packings in projective spaces.
These papers were the motivation for defining the constants $c_{n,k}$ in~\cite{GH-BU-VR}, which furthered the study of Gromov--Hausdorff distances between spheres.

The weak equivalence \eqref{eq:vr->vrm-weak-equiv} between Vietoris--Rips complexes and metric thickenings (with a ``less than'' convention) implies that the constant $c_{n,k}$ could alternatively be defined with $\vrm{S^n}{r}$ in place of $\vr{S^n}{r}$.
More generally, if a finite group $G$ acts on a CW complex $X$ by cellular maps and on a metric space $Y$ by isometries, then we have
\begin{equation} \label{eq:inf-vr-and-p}
    \inf\{r \ge 0 \colon~ \exists~ \text{continuous~} X \to_G \vr{Y}{r}\} = \inf\{r \ge 0 \colon~ \exists~ \text{continuous~} X \to_G \vrm{Y}{r}\}.
\end{equation}
We point out that the constant on the left hand side was described to us by Facundo M\'emoli in a Polymath meeting in 2022 (see also \cite[Definition~24]{lim2025g}). Indeed, the infimum on the left hand side is not smaller due to the $G$-map $\vr{X}{r} \to \vrm{X}{r}$ induced by the identity.
As for the other inequality, if there is a $G$-map $f \colon~ X \to \vrm{Y}{r}$, it can be lifted (up to $G$-equivariant homotopy) to a $G$-map $\widetilde{f} \colon~ X \to \vrless{Y}{r+\varepsilon} \subseteq \vr{Y}{r+\varepsilon}$, for any $\varepsilon >0$, as depicted in the following diagram of $G$-maps (which commutes up to $G$-homotopy):

\begin{equation*}
        \begin{tikzcd}
            {} & {} &  \vrless{Y}{r+\varepsilon} \arrow[d, "\sim"]\\
            X \arrow[r, "f"] \arrow[urr, dashed,  bend left=18, "\widetilde{f}"] &  \vrm{Y}{r} \arrow[r] & \vrmless{Y}{r+\varepsilon}.
        \end{tikzcd}
    \end{equation*}
    For $G$-spaces $A$ and $B$, we will denote by $[A,B]_G$ the set of $G$-maps $A \to B$ up to $G$-equivariant homotopy equivalence.
Then, the composition of the two horizontal arrows in the diagram defines an element $[f] \in [X, \vrmless{Y}{r+\varepsilon}]_G$.
By \cite[Proposition II.2.6]{tom1987transgroups}, the vertical arrow in the diagram, which is a weak homotopy equivalence and a $G$-map, induces a surjection
    \begin{equation*}
        [X, \vrless{Y}{r+\varepsilon}]_G \longrightarrow [X, \vrmless{Y}{r+\varepsilon}]_G,
    \end{equation*}
    which shows that there is an element $[\widetilde{f}] \in [X, \vrless{Y}{r+\varepsilon}]_G$ mapping to $[f]$. The equality \eqref{eq:inf-vr-and-p} was also obtained in a different way by explicitly constructing maps between Vietoris-Rips complexes and metric thickenings (see \cite[Proposition~6.2]{lim2025g}). It is also possible to obtain a generalization of (\ref{eq:oddmod}) using the quantities in \eqref{eq:inf-vr-and-p} to the case when $G$ is a finite group and $f \colon X \to Y$ is any $G$-function between $G$-metric spaces. We do this in special cases of interest (see proofs of Theorems \ref{thm:quant-general}, \ref{thm:quant-vk-f}, and \ref{theorem: discontinuous tverberg}); for a general case, see \cite[Proposition 6.5]{lim2025g}.
    
    Our paper and the recent paper of Lim and M\'emoli \cite{lim2025g} are follow-up projects to the polymath-style paper \cite{GH-BU-VR}; as such, the two papers use related methods towards different goals. The main goal of \cite{lim2025g} is to define equivariant distances, obtain bounds, and establish equivariant rigidity results.

%=============

\subsection*{Coindex}

For a $\Z/2$-space $X$, its $\Z/2$-\emph{coindex} is defined as the integer
\begin{equation*}
    \coind_{\Z/2}(X) = \max\{k \in \Z_{\ge 0} \colon~ \text{$\exists$ an odd map $S^{k} \longrightarrow X$}\}.
\end{equation*}
It is monotone in the sense that if there exists a $\Z/2$-map $X\to Y$, then $\coind_{\Z/2}(X) \le \coind_{\Z/2}(Y)$.
In general, if $X$ is $(k-1)$-connected, then $\coind_{\Z/2}(X) \ge k$.
Moreover, the Borsuk--Ulam theorem yields that $\coind_{\Z/2}(S^k) = k$, where $\Z/2$ is assumed to act antipodally on $S^k$.
We refer the reader to Matou\v{s}ek's book \cite[Section~5]{matousek2003using} for more details.
In the language of coindex, we can write the constants $c_{n,k}$ as $\inf\{r\colon~ \coind_{\Z/2}(\vr{S^n}{r} \ge k)\}$.

%=============

\section{Injective functions}
\label{sec:injectivity}

Given spaces $X$ and $Y$, an \emph{embedding} of $X$ into $Y$ is a map $f \colon X \to Y$ that is a homeomorphism onto its image.
Assuming $X$ is compact, this is equivalent to the map $f$ being injective.
If such a map $f$ exists, we say that $X$ \emph{embeds} into $Y$.
A fundamental problem in geometric topology is to determine, given a space $X$, the minimum dimension $d$ such that $X$ embeds into $\R^{d}$.
Classical results on this problem include the characterization of planar graphs (as graphs with no $K_5$ or $K_{3,3}$ minors)~\cite{kuratowski1930probleme, wagner1937eigenschaft}, Whitney's embedding theorem (that any smooth $d$-manifold embeds into~$\R^{2d}$)~\cite{whitney44}, the van Kampen--Flores theorem (that the $d$-skeleton of the $(2d+2)$-simplex does not embed into~$\R^{2d}$)~\cite{van1933komplexe}, and upper and lower bounds for the embedding dimension of real projective space~$\R P^n$; see~\cite{james1963immersion, sanderson1964immersions}.
Given a nonembeddability result of the form, ``$X$ does not embed into $\R^{d}$,'' we aim to quantify the lack of embeddability by asking, given an injective function $X \to \R^d$, how discontinuous must it be?

\medskip

The modulus of discontinuity is inadequate in this context, since there exist injective functions $f \colon \R^{k} \to \R^{d}$ with $\text{diam}(\text{range}(f)) < \varepsilon$ for arbitrary $k, d, \varepsilon$, even if $d < k$.
Hence given a space $X$ that embeds into any Euclidean space $\R^{k}$ (say, via $g \colon X \to \R^{k}$), the function $f \circ g \colon X \to \R^{d}$ is injective, and $\delta(f \circ g) \le \varepsilon$; that is, the modulus of discontinuity of injective functions $X \to \R^{d}$ can be made arbitrarily small.
This suggests that we need a different measure of discontinuity.

%=============

\subsection{Scale-invariant modulus of discontinuity $\alpha$}

\medskip

Let us denote by
\[
    \Conf_2(X) = \{(x,y) \in X \times X \colon ~ x \ne y\}
\]
the \emph{configuration space} of two distinct points in $X$,
equipped with the $\Z/2$-action which interchanges $x$ and $y$.
Given a function $f\colon X \to \R^d$, we consider the $\Z/2$-equivariant test function $f(x) - f(y)$, defined on $\Conf_2(X)$.
This test function detects injectivity in the sense that its image contains $0$ if and only if $f$ is not injective.
In particular, an injective function $f \colon X \to \R^d$ yields a $\Z/2$-equivariant map
\begin{equation}
\label{eq:test-inj}
\Phi_f \colon \Conf_2(X) \longrightarrow S^{d-1},~(x,y) \longmapsto \frac{f(x)-f(y)}{\| f(x) - f(y) \|}.
\end{equation}
Hence nonembeddability results may be obtained by obstructing the existence of $\Z/2$-equivariant maps $\Phi \colon \Conf_2(X) \to S^{d-1}$, for example by using the Borsuk--Ulam theorem.
On the other hand, in some settings the existence of $\Phi$ implies the existence of $f$, for example the following result by Haefliger in the smooth case and Weber in the simplicial complex case:

\begin{theorem}[Haefliger~\cite{Haefliger}, Weber~\cite{Weber}]
\label{thm:Haefliger}
Let $X$ be a smooth, closed manifold (resp.\ simplicial complex) of dimension $n$.
If $d > \frac32(n + 1)$, then there exists a differentiable (resp.\ linear) embedding $f \colon X \to \R^d$ if and only if there exists a $\Z/2$-equivariant map $\Phi \colon \Conf_2(X) \to S^{d-1}$.
\end{theorem}

These results establish a strong relationship between injectivity of $f$ and existence of $\Phi$.
Motivated by this, we define the following measure of discontinuity, illustrated in Figure \ref{fig:alpha}.

\begin{definition}
\label{def:alpha}
Let $X$ be a topological space, let $f \colon X \to \R^d$ be injective, and let $\Phi_f \colon \Conf_2(X) \to S^{d-1}$ be the induced function defined above in (\ref{eq:test-inj}).
The \emph{scale-invariant modulus of discontinuity} $\alpha(f)$ is defined as the modulus of discontinuity of $\Phi_f$; that is, $\alpha(f) = \delta(\Phi_f)$.
\end{definition}

\begin{figure}
    \centering
    \includegraphics[width=0.70\linewidth]{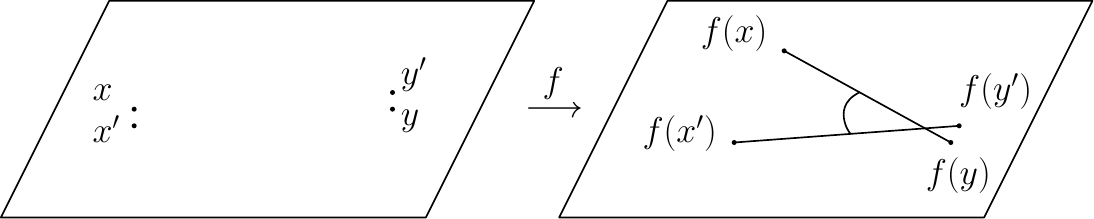}
    \caption{Two close pairs $(x,y), (x',y') \in \Conf_2(X)$ which induce a large angle between $f(x)-f(y)$ and $f(x')-f(y')$.}
    \label{fig:alpha}
\end{figure}

Our first general result is the following.

\begin{theorem} \label{thm:quant-general}
    Let $X$ compact topological space.
Assume that there does not exist a $\Z/2$-equivariant map $\Conf_2(X) \to S^{d-1}$.
Then, every injective function $f \colon X \to \R^d$ satisfies $\alpha(f) \ge r_{d-1} = \arccos{(-1/d)}$.
\end{theorem}
\begin{proof}
    Let $f \colon~ X \to \R^{d}$ be an injective function.
A straightforward generalization of \cite[Lemma~7.4]{GH-BU-VR} to metric thickenings implies that for any $\varepsilon > 0$ there exists a parameter $\rho = \rho(\varepsilon)$ such that the function
    \begin{equation*}
        \vrm{\Conf_2(X)}{\rho} \longrightarrow \vrm{S^{d-1}}{\delta(\Phi_f) + \varepsilon}
    \end{equation*}
    induced from $\Phi_f$ is continuous and $\Z/2$-equivariant.
As noted in Section \ref{sec:background}, there is a canonical $\Z/2$-inclusion $\Conf_2(X) \hookrightarrow \vrm{\Conf_2(X)}{\rho}$, which then implies the existence of a $\Z/2$-map
    \begin{equation}\label{eq:conf-to-P} 
        \Conf_2(X) \longrightarrow \vrm{S^{d-1}}{\delta(\Phi_f) + \varepsilon}.
    \end{equation}
    If $\delta(\Phi_f)  + \varepsilon < r_{d-1}$, there would exist a $\Z/2$-map $\vrm{S^{d-1}}{\delta(\Phi_f) + \varepsilon} \to S^{d-1}$ from \eqref{eq:P(Sn)-homotopy}, which precomposed with \eqref{eq:conf-to-P} would yield a $\Z/2$-map
    \begin{equation*}
        \Conf_2(X) \longrightarrow S^{d-1}.
    \end{equation*}
    This is a contradiction with our assumption, hence $\delta(\Phi_f)  + \varepsilon \ge r_{d-1}$.
Letting $\varepsilon \to 0$ we obtain the claim $\alpha(f) = \delta(\Phi_f) \ge r_{d-1}$.
\end{proof}

Combining this with the Haefliger--Weber theorem, we obtain a quantified version of their results.

\begin{corollary}\label{cor:quant-hw}
    Let $X$ be a smooth, closed manifold (resp.\ simplicial complex) of dimension $n$.
Assume that there does not exists a differentiable (resp.\ linear) embedding $f \colon X \to \R^d$ for $d > \frac32(n + 1)$.
Then, every injective function $f \colon X \to \R^d$ satisfies $\alpha(f) \ge r_{d-1} = \arccos{(-1/d)}$.
\end{corollary}

Theorem \ref{thm:quant-general} works under the assumption of the nonexistence of a $\Z/2$-map $\Conf_2(X) \to S^{d-1}$ and provides a general lower bound for $\alpha$.
However, we can obtain an improved bound by understanding the ``degree'' of the failure of the existence of the $\Z/2$-map.
The lower bound depends on the topology of Vietoris--Rips complexes of spheres (see Section \ref{sec:background}).

\begin{theorem}
\label{thm:alpha-coind}
Let $X$ be a topological space, let $k = \operatorname{coind}_{\Z/2}(\Conf_2(X))$, and assume $k \ge d - 1$.
Then any injective function $f \colon X \to \R^d$ satisfies $\alpha(f) \geq c_{d-1,k}$.
\end{theorem}
\begin{proof}
By the assumption on coindex, there exists a $\Z/2$-map $g \colon~ S^k \to \Conf_2(X)$.
Note the composition $\Phi_f\circ g \colon S^k\to S^{d-1}$ is a $\Z/2$-map.
Therefore by~\eqref{eq:oddmod}, $\delta(\Phi_f \circ g) \geq c_{d-1,k}$.
Finally, we obtain the claim since $\alpha(f)=\delta(\Phi_f) \ge \delta(\Phi_f \circ g)$ by \cite[Lemma 7.10(ii)]{GH-BU-VR}.
\end{proof}

\begin{corollary}
Assume $k \ge d - 1$.
Then we have the following results:
\begin{itemize}
\item 
Any injective function $f \colon \R^{k + 1} \to \R^d$ satisfies $\alpha(f) \ge c_{d - 1, k}$.
\item
Any injective function $f \colon S^{k} \to \R^d$ satisfies $\alpha(f) \ge c_{d - 1, k}$.
\end{itemize}
\end{corollary}

\begin{proof}
We define a chain of $\Z/2$-maps:
\begin{equation*}
S^{k} \longrightarrow \Conf_2(S^{k}) \longrightarrow \Conf_2(\R^{k + 1}) \longrightarrow S^{k}.
\end{equation*}
The individual $\Z/2$-maps are defined as follows:
\begin{itemize}
\item  The first map $S^{k} \to \Conf_2(S^{k})$ is defined by $x \mapsto (x, -x)$.
\item  The second map $\Conf_2(S^{k}) \to \Conf_2(\R^{k + 1})$ is induced by the inclusion $S^{k} \subseteq \R^{k + 1}$.
\item The third map $\Conf_2(\R^{k + 1}) \to S^{k}$ is defined by $(x, y) \mapsto \frac{x - y}{\|x - y\|}$.
\end{itemize}
By the Borsuk--Ulam theorem, this shows $\operatorname{coind}_{\Z/2}(\Conf_2(S^{k})) = \operatorname{coind}_{\Z/2}(\Conf_2(\R^{k + 1})) = k$, so the corollary follows from Theorem \ref{thm:alpha-coind}.
\end{proof}

Note that if $k + 1 > d$, then $c_{d - 1, k} > 0$, so $\Phi_{f}$ is discontinuous and hence so is $f$.
Therefore, this corollary represents a generalization of the known nonembeddability results on Euclidean spaces.
%Moreover, 

%=============

\subsection{Suitability of $\alpha$ for detecting discontinuity} \label{sec:suitability-alpha}
If an injective function $f\colon X \to \R^d$ is continuous, then so is $\Phi_f$, and hence $\alpha(f) = \delta(\Phi_f) = 0$.
We will show that under a mild assumption, a converse statement also holds:

\begin{theorem}
\label{thm:noline}
Let $f \colon X \to \R^d$ be an injective function.
Suppose that the image $f(X)$ is not contained in an affine line in $\R^d$.
Then, $\alpha(f) = 0$ if and only if $f$ is continuous.
\end{theorem}

We may interpret the theorem as follows.
We know that $\alpha(f) = \delta(\Phi_f) = 0$ if and only if $\Phi_f$ is continuous.
When does continuity of $\Phi_f$ imply continuity of $f$?
Continuity of $\Phi_f$ implies that for fixed $y \in X$, the restriction of $\Phi_f$ given by $(\Phi_f)_y \colon X - \left\{ y \right\} \to S^{d-1} \colon x \mapsto \Phi_f(x,y)$ is continuous.
The function $(\Phi_f)_y$ can see the direction of $f(x)-f(y)$ but not the distance $\|f(x)-f(y)\|$.
Thus $(\Phi_f)_y$ can detect a discontinuity of $f$ at a point $x$ as long as the discontinuity does not manifest on the line through $f(y)$ and $f(x)$.
In particular, as long as the image of $f$ is not contained in a line, any discontinuity of $f$ can be witnessed by some $y$.

Moreover, the assumption on the image of $f$ from Theorem  \ref{thm:noline} is fairly mild.
For example, if the domain $X$ equals $S^1$ or is a manifold/CW complex of dimension at least two, then the condition is automatically satisfied, as shown by the following.

\begin{corollary}
\label{cor:dim2}
Let $f \colon X \to \R^d$ be an injective function.
Suppose there exists a map $\psi \colon~ S^1 \to X$ such that $\psi(x) \neq \psi(-x)$ for all $x \in S^1$.
Then $\alpha(f) = 0$ if and only if $f$ is continuous.
\end{corollary}

\begin{proof}
In order to apply Theorem~\ref{thm:noline}, we will show that it cannot happen that the image of $f$ is contained in an affine line, that $\alpha(f) = 0$, and that there is a map $\psi \colon~ S^1 \to X$ with $\psi(x) \neq \psi(-x)$ for all $x \in X$.
Indeed, if the image of $f$ is contained in an affine line, then $f$ may be considered as a function $X \to \R$.
If additionally $0 = \alpha(f) = \delta(\Phi_f)$, then $\Phi_f \colon \Conf_2(X) \to S^0$ is a continuous $\Z/2$-map.
Finally, $\psi$ induces a $\Z/2$-map $\Psi \colon~ S^1 \to \Conf_2(X)$ by $x \mapsto (\psi(x), \psi(-x))$, and a $\Z/2$-map $\Phi_f \circ \Psi \colon~ S^1 \to S^0$ contradicts the Borsuk--Ulam theorem.
\end{proof}

%By Theorem~\ref{thm:noline} and Corollary~\ref{cor:dim2}, $\alpha$ detects discontinuity of injective functions $f$ as long as either the domain or image is at least $2$-dimensional.

The following lemma shows that the scale-invariant modulus of discontinuity fails to detect discontinuities of functions from general 1-dimensional domains; for example, a monotone function $f \colon \R \to \R$ with jump discontinuities still has $\alpha(f) = 0$:

\begin{lemma}
Let $f \colon \R \to \R$ be injective.
Then, $\alpha(f) = 0$ if and only if $f$ is monotone.
\end{lemma}

\begin{proof}
$\Conf_2(\R)$ has two connected components: $\{(x,y) \in \R^2 \colon~x>y\}$ and $\{(x,y) \in \R^2 \colon~x<y\}$.
Therefore, a $\Z/2$-map $\Phi_f \colon \Conf_2(\R) \to S^0$ is continuous if and only if it sends one entire component to $+1$ and the other to $-1$.
The latter means that the sign of $f(x)-f(y)$ is constant on each of the components, i.e., $f$ is monotone.
\end{proof}

\begin{proof}[Proof of Theorem~\ref{thm:noline}]
As remarked at the beginning of this subsection, if $f$ is continuous, then $\alpha(f)=0$.
We will show that if $f$ is not continuous, then $\alpha(f) > 0$.
Let $x$ be a point of discontinuity of $f$.
Then there exists a sequence $x_n \to x$ such that the collection of images $f(x_n)$ lie outside some $\varepsilon_1$-neighborhood of $f(x)$.
The elements
\[
u_n \coloneqq \frac{f(x)-f(x_n)}{\| f(x) - f(x_n) \|} \in S^{d-1}
\]
lie in a compact space, so there exists a  subsequence (also labeled $x_n$) such that $u_n$ converges to some $u \in S^{d-1}$.
This direction $u$ determines an affine line $\ell = \{f(x) + \lambda u : \lambda \in \mathbb{R}\} \subseteq \R^d$, which we might intuitively consider as the line on which the discontinuity manifests.
%-----
\begin{comment}
Geometrically, this convergence means that for any small angle $\varepsilon_2$, there exists $N$ such that for all $n > N$, $f(x_n)$ lies inside the cone
\[
C_{\varepsilon_2} \coloneqq \{ f(x) + \lambda u' : \lambda > 0, u' \in B_{\varepsilon_2}(u) \subset S^{d-1}\} \subset \R^d;
\]
this is the union of rays emanating from $f(x)$ with directions from $B_{\varepsilon_2}(u) \subset S^{d-1}$.
In fact, by the hypothesis on $x_n$, all such $f(x_n)$ lie inside the severed cone
\[
C_{\varepsilon_1,\varepsilon_2} \coloneqq \{ f(x) + \lambda u' : \lambda \geq \varepsilon_1, u' \in B_{\varepsilon_2}(u) \subset S^{d-1}\} \subset \R^d.
\]
Now, by the assumption that $f(X)$ is not contained in an affine line, we may choose some $y\in X$ with $f(y) \notin \ell$, to ``witness" the discontinuity of $f$ at $x$.
Since $y \notin \ell$, we may choose our angle $\varepsilon_2$ such that $y \notin C_{\varepsilon_2}$; this means that the ray from $f(x)$ in the direction of $f(y)$ does not intersect the closed set $C_{\varepsilon_1, \varepsilon_2}$.
Therefore there exists $\varepsilon > 0$ such that 
\[
\left\langle \frac{f(x)-f(y)}{\| f(x) - f(y) \|}, \frac{f(x_{n})-f(y)}{\| f(x_{n}) - f(y) \|} \right\rangle < 1 - \varepsilon
\]
for sufficiently large $n$, and $\alpha(f) > \arccos(1 - \varepsilon) > 0$, as desired.
\end{comment}
%-----
Now, by the assumption that $f(X)$ is not contained in an affine line, we may choose some $y\in X$ with $f(y) \notin \ell$, to ``witness" the discontinuity of $f$ at $x$.
See Figure \ref{fig:continuity} for illustration.

If the sequence of points $f(x_n)$ has a bounded subsequence, then it has a subsequence (also denoted by $f(x_n)$) that satisfies $f(x_n) \to z$, for some $z \in \ell$.
Let
\[
    v \coloneqq \frac{z-f(y)}{\|z-f(y)\|} = \lim_{n \to \infty} \frac{f(x_n)-f(y)}{\|f(x_n)-f(y)\|} \in S^{d-1} .
\]
From $\|f(x)-f(x_n)\| \ge \varepsilon_1$ we obtain $\|f(x)-z\| \ge \varepsilon_1$, which implies $v \neq u$.
Therefore there exists $\varepsilon > 0$ such that
\begin{equation} \label{eq:angle}
    \left\langle \frac{f(x)-f(y)}{\| f(x) - f(y) \|}, \frac{f(x_{n})-f(y)}{\| f(x_{n}) - f(y) \|} \right\rangle < 1 - \varepsilon
\end{equation}
for sufficiently large $n$, and $\alpha(f) > \arccos(1 - \varepsilon) > 0$, as desired.

On the other hand, if the sequence $f(x_n)$ does not have a bounded subsequence, the vector
\[
    v = \lim_{n \to \infty} \frac{f(x_n)-f(y)}{\|f(x_n)-f(y)\|} \in S^{d-1}
\]
equals $-u$.
In particular, it again satisfies $v \neq u$, so the claim follows similarly from \eqref{eq:angle}.
\end{proof}

\begin{figure}[h]
    \centering
    \includegraphics[width=0.55\linewidth]{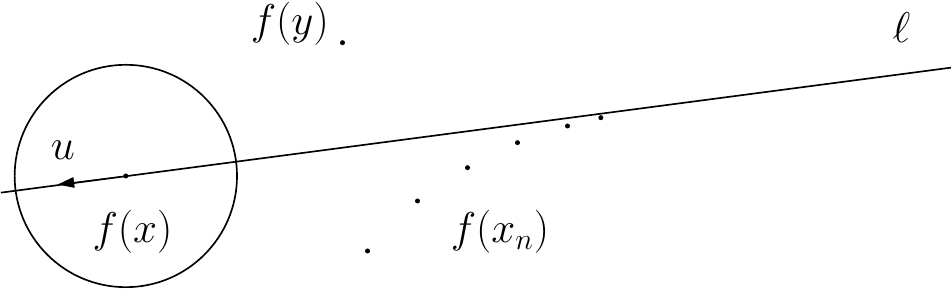}
    \caption{Discontinuity of $f$ at $x$.}
    \label{fig:continuity}
\end{figure}

% \begin{proof}[Proof of Theorem~\ref{thm:noline}]
% As remarked at the beginning of this subsection, if $f$ is continuous, then $\alpha(f)=0$.
% We will show %the contrapositive, 
% that if $f$ is not continuous, then $\alpha(f) > 0$.
% Let $x$ be a point of discontinuity of $f$.
% By discontinuity at $x$ and compactness of $S^{d-1}$, there exists a sequence $x_n \to x$ such that
% \[
% \frac{f(x)-f(x_n)}{\| f(x) - f(x_n) \|} \to u \in S^{d-1}.
% \]
% This direction $u$ determines an affine line $\ell = \{f(x) + \lambda u : \lambda \in \mathbb{R}\} \subseteq \R^d$.
% By the assumption that $f(X)$ is not contained in an affine line, we may choose some $y\in X$ with $f(y) \notin \ell$.
% There exists a subsequence $\left\{x_{n_k}\right\}$ such that
% \[
% \frac{f(x_{n_k})-f(y)}{\| f(x_{n_k}) - f(y) \|} \to v \in S^{d-1}.
% \]
% By the hypothesis $f(y) \notin \ell$, we have
% \[
% \left\langle \frac{f(x)-f(y)}{\| f(x) - f(y) \|}, v \right\rangle < 1 - \varepsilon
% \]
% for some $\varepsilon > 0$, and hence
% \[
% \left\langle \frac{f(x)-f(y)}{\| f(x) - f(y) \|}, \frac{f(x_{n_k})-f(y)}{\| f(x_{n_k}) - f(y) \|} \right\rangle < 1 - \varepsilon/2
% \]
% for sufficiently large $n_k$.
% Therefore $\alpha(f) > \arccos(1 - \varepsilon/2) > 0$, as desired.
% \end{proof}

%=============

\subsection{Van Kampen--Flores theorem} \label{subsec:vk-f}
A simplicial complex of dimension $d$ can always be embedded (even affinely) in $\R^{2d+1}$, for example, by placing the vertices of the complex along the moment curve $\R \to \R^{2d+1},~t \mapsto (t, t^2, \dots, t^{2d+1})$.
The result of van Kampen \cite{van1933komplexe} and Flores \cite{flores1933n} states that the $d$-skeleton of the $(2d+2)$-dimensional simplex does not embed in $\R^{2d}$, showcasing that the above general embedding result is the best possible.
In fact, they prove a stronger statement.
For a simplicial complex $K$, we say that a function $f \colon K \to \R^n$ is \emph{almost injective} if $f(x_1) = f(x_2)$ implies that $x_1$ and $x_2$ belong to a common face of~$K$.
This describes a wider class of functions than just injective functions.
A function that is almost injective and is continuous will be called an \emph{almost embedding}.

\begin{theorem}[van Kampen--Flores \cite{flores1933n,van1933komplexe}]
    \label{thm:vk-f}
    Let $d \ge 1$ be an integer.
Then, there does not exists an almost embedding $\sk_d(\Delta_{2d+2}) \to \R^{2d}$.
\end{theorem}

As before, a natural question arises: \emph{Since every almost injective function $\sk_d(\Delta_{2d+2}) \to \R^{2d}$ is discontinuous, can we quantify how discontinuous it needs to be}?
To answer this question, we will need to adjust the notion of the scale invariant modulus of discontinuity $\alpha$.
Namely, for a simplicial complex $K$, let
\begin{equation*}
    \Conf_2^{\Delta}(K) \coloneqq \{(x,y) \in K \times K \colon~ x,y~\text{belong to disjoint faces}\}
\end{equation*}
denote the appropriate configuration space, in the literature \cite{matousek2003using} also known as the \emph{2-fold deleted product of $K$}.
We again equip it with a $\Z/2$-action that flips $x$ and $y$.
Then, an almost embedding $f \colon~ K \to \R^d$ induces a $\Z/2$-map
\begin{equation} \label{eq:Pf-vk-f}
    P_f \colon \Conf_2^{\Delta}(K) \longrightarrow S^{2d-1},~(x,y) \longmapsto \frac{f(x)-f(y)}{\|f(x)-f(y)\|}, 
\end{equation}
which is the restriction of the map \eqref{eq:test-inj}.
We now define the \emph{scale invariant modulus of discontinuity of an almost injective function} $f$ as $\alpha^{(2)}(f) \coloneqq \delta(P_f)$.
The reason we have the number two in the superscript of $\alpha$ is that in Section \ref{sec:almost-r-emb} we will study the more general almost $r$-injective functions and define the appropriate modulus of discontinuity $\alpha^{(r)}$.

With the new notion at hand, we can prove the following generalization of Theorem \ref{thm:vk-f}.

\begin{theorem}[Quantified van Kampen--Flores]
    \label{thm:quant-vk-f}
    Let $d \ge 1$ be an integer.
Then, every almost injective function $\sk_d(\Delta_{2d+2}) \to \R^{2d}$ satisfies $\alpha^{(2)}(f) \ge r_{2d-1} = \arccos{\left(\frac{-1}{2d}\right)}$.
\end{theorem}
\begin{proof}
    Let $f \colon~ \sk_d(\Delta_{2d+2}) \to \R^{2d}$ be almost injective.
Analogously as in the proof of Theorem \ref{thm:quant-general}, the function $P_f$ from \eqref{eq:Pf-vk-f} induces a $\Z/2$-map
    \begin{equation} \label{eq:del-prod-to-sphere}
        \Conf_2^{\Delta}(\sk_d(\Delta_{2d+2})) \longrightarrow \vrm{S^{2d-1}}{\alpha^{(2)}(f) + \varepsilon}.
    \end{equation}
    Let us show that this can only happen if $\alpha^{(2)}(f) + \varepsilon \ge r_{2d-1}$.
    Indeed, if $\alpha^{(2)}(f) + \varepsilon < r_{2d-1}$, the homotopy equivalence \eqref{eq:P(Sn)-homotopy} (and the text thereafter) would yield a $\Z/2$-map 
    \[
        \vrm{S^{2d-1}}{\alpha^{(2)}(f) + \varepsilon} \longrightarrow S^{2d-1}
    \] which when precomposed with \eqref{eq:del-prod-to-sphere} would yield a $\Z/2$-map
    \begin{equation*}
        F \colon \Conf_2^{\Delta}(\sk_d(\Delta_{2d+2})) \longrightarrow S^{2d-1}.
    \end{equation*}
    Let us introduce the \emph{2-fold deleted join} (see \cite[Section~5.5]{matousek2003using} for more detail) as a simplicial complex
    \begin{equation*}
        (\sk_d(\Delta_{2d+2}))^{*2}_{\Delta} \coloneqq \{tx + (1-t)y \in (\sk_d(\Delta_{2d+2}))^{*2} \colon~ x,y~\text{belong to disjoint faces}\}.
    \end{equation*}
    The deleted product can be seen as a subspace of the deleted join via the inclusion
    \begin{equation*}
        \Conf_2^{\Delta}(\sk_d(\Delta_{2d+2})) \longrightarrow (\sk_d(\Delta_{2d+2}))^{*2}_{\Delta},~(x,y) \longmapsto \tfrac{1}{2}x + \tfrac{1}{2}y.
    \end{equation*}
    The map $F$ can be extended to a $\Z/2$-map
    \begin{equation} \label{eq:extended-F}
        (\sk_d(\Delta_{2d+2}))^{*2}_{\Delta} \longrightarrow S^{2d},\quad tx + (1-t)y \longmapsto \frac{(1-2t,t(1-t)F(x,y))}{\|(1-2t,t(1-t)F(x,y))\|}.
    \end{equation}
    However, this is a contradiction: by \cite[Section~5.6, pg.~117]{matousek2003using} the complex $(\sk_d(\Delta_{2d+2}))^*_{\Delta}$ is $\Z/2$-homeomorphic to a $(2d+1)$-sphere (in fact, it is a \emph{Bier sphere}; for more details see~\cite[Definition~5.6.1~and~Theorem~5.6.2]{matousek2003using}), so the $\Z/2$-map in \eqref{eq:extended-F} cannot exist due to the Borsuk-Ulam theorem~\cite{borsuk1933drei}.
\end{proof}

We remark that we cannot improve the lower bound on the scale parameter for the non-existence of a $\Z/2$-map \eqref{eq:del-prod-to-sphere} since a $\Z/2$-map 
    \begin{equation*}
        \Conf_2^{\Delta}(\sk_d(\Delta_{2d+2})) \longrightarrow \vrm{S^{2d-1}}{r_{2d-1}}
    \end{equation*} exists.
Namely, the domain is a $(2d)$-dimensional cellular complex with a free $\Z/2$-action and by \eqref{eq:P(Sn)-homotopy} the codomain is $(2d-1)$-connected.
Hence, all obstructions to defining such a $\Z/2$-map vanish (see also \cite[Lemma~6.2.2]{matousek2003using}).

%=============

\section{Almost \texorpdfstring{$r$}{r}-embeddings}
\label{sec:almost-r-emb}

In this section we will prove a quantified topological Tverberg's theorem \cite{BSSz1981,ozaydin1987} using an $r$-fold version of $\alpha^{(2)}$.
We will then use this to obtain bounds on the quotient of modulus of discontinuity by an $r$-fold version $\kappa^{(r)}$ of the Gromov--Guth thickness~\cite{gromov2012generalizations}; see also Barzadin~\cite{barzdin1993realization}.

\subsection{Quantified topological Tverberg theorem}

\begin{definition}
    Let $r \ge 2$, $d \ge 0$ be integers and $K$ a simplicial complex.
A function $f\colon K \to \R^d$ is \emph{almost $r$-injective} if any $r$ pairwise disjoint faces $\sigma_1,\dots ,\sigma_r \subseteq \Delta_N$ satisfy $f(\sigma_1) \cap \dots \cap f(\sigma_r) = \emptyset$.
\end{definition}

In another words, an almost $r$-injective function $f$ never identifies $r$ points from pairwise disjoint faces.
This represents the $r$-fold generalization of the notion of the almost injectivity ($r=2$) from the previous section.
Tverberg \cite{tverberg1966} proved that for $r \ge 2$ and $d \ge 0$ any \emph{affine} function $f\colon \Delta_N \to \R^d$ is not almost $r$-injective if the dimension of the simplex is large enough, namely if $N \ge (r-1)(d+1)$.
B\'ar\'any, Shlosman and Sz\H ucs \cite{BSSz1981} and \"Ozaydin \cite{ozaydin1987} showed that no such \emph{continuous} almost $r$-injective function exists if $r$ is a prime power.

\begin{theorem}[Topological Tverberg, prime power case {\cite{BSSz1981, ozaydin1987}}] \label{theorem: topological tverberg}
    Let $d \ge 1$ be an integer, let $r \ge 2$ be a prime power, and let $N = (r-1)(d+1)$.
Let $f\colon \Delta_N \to \R^d$ be any function.
Then, if $f$ is $r$-injective, it is not continuous.
\end{theorem}

The prime power condition is really necessary: Frick \cite{frick2015counterexamples} and Blagojevi\'c, Frick and Ziegler~\cite{blagojevic2015barycenters}, building on work of Mabillard and Wagner \cite{mabillard2014eliminating}, showed that if $r$ is not a prime power and $d \ge 3r$, then a continuous almost $r$-embedding $f \colon~ \Delta_N \to \R^d$ exists.
The lower bound on $d$ was improved to $d \ge 2r+1$ by Avvakumov, Mabillard, Skopenkov and Wagner \cite{avvakumov2021eliminating}.
In particular, when $r$ is a prime power, a natural question of quantifying discontinuity of almost $r$-injective functions emerges.

In order to answer it, in Definition \ref{def:alpha^(r)} below, we introduce a notion $\alpha^{(r)}$ that captures discontinuity of such maps.
It is an $r$-fold analogue of scale-invariant moduli of discontinuity $\alpha$ and $\alpha^{(2)}$ introduced in Section \ref{sec:injectivity}.
The following is the main theorem of the section.

\begin{theorem} [Quantified topological Tverberg] \label{theorem: discontinuous tverberg}
    Let $d \ge 1$ be an integer, $r$ a prime power, and $f\colon \Delta_{(r-1)(d+1)} \to \R^d$ any function.
Then, if $f$ is almost $r$-injective it must satisfy
    \[
    	\alpha^{(r)}(f) ~\ge~  \arccos(\tfrac{-1}{d(r-1)}).
    \]
\end{theorem}

The constant $\arccos(\tfrac{-1}{d(r-1)})$ denotes the diameter of the regular $d(r-1)$-dimensional simplex inscribed in the $(d(r-1)-1)$-sphere equipped with the geodesic metric.
As explained in Remark \ref{remark: R_delta and angles}, the fact that $\alpha^{(r)}(f) > 0$ implies that $f$ is not continuous.
Therefore, Theorem \ref{theorem: discontinuous tverberg} presents a generalization of the topological Tverberg theorem.

\begin{definition} \label{def:r-del-prod}
	For an integers $r \ge 2$ and $N \ge r-1$ we define a cellular complex
	\[
        \Conf_r^{\Delta}(\Delta_N) := \{(x_1,...,x_r) \in  (\Delta_N)^{\times r}: x_i\text{'s belong to pairwise disjoint faces}\},
	\]
    called the \textit{$r$-fold deleted product} of $\Delta_N$.
\end{definition}

The symmetric group $\mathfrak{S}_r$ acts freely on $\Conf_r^{\Delta}(\Delta_N)$ by permuting the coordinates.
Let us denote by
\(
	W_r^{\oplus d} := \{(z_1,...,z_r) \in (\R^d)^{\oplus r}: z_1 + \dots + z_r = 0\}
\)
a $d(r-1)$-dimensional $\mathfrak{S}_r$-representation, where the action is also given by permutation of coordinates.
Next, we define an $\mathfrak{S}_r$-map
\begin{equation*} \label{eq:map Pf}
	\Conf_r^{\Delta}(f) \colon \Conf_r^{\Delta}(\Delta_N) \longrightarrow W_r^{\oplus d}
\end{equation*}
which sends a tuple $(x_1,...,x_r) \in \Conf_r^{\Delta}(\Delta_N)$ to a tuple obtained from 
\(
	(f(x_1), \dots , f(x_r)) \in (\R^d)^{\oplus r}
\)
by subtracting the average $\frac{1}{r}(f(x_1) + \dots + f(x_r))$ from each of the $r$ coordinates $f(x_i)$.
We note that $f$ is almost $r$-injective if and only if the image of $\Conf_r^{\Delta}(f)$ does not contain the origin $0 \in W_r^{\oplus d}$, which enables the following definition.

\begin{definition} \label{def:alpha^(r)}
	Let $f\colon \Delta_N \to \R^d$ be an almost $r$-embedding.
Then we define
	\[
		\alpha^{(r)}(f) := \delta\Big(\Conf_r^{\Delta}(\Delta_N) ~ \xrightarrow{~\Conf_r^{\Delta}(f)~}~ W_r^{\oplus d} \setminus \{0\} ~\xrightarrow{~\nu~} ~S(W_r^{\oplus d})\Big),
	\]
	where $\nu\colon v \mapsto v/\|v\|$ is a deformation retraction and the unit sphere $S(W_r^{\oplus d}) \subseteq W_r^{\oplus d}$ is assumed to be endowed with the geodesic metric.
\end{definition}

\begin{remark} \label{remark: R_delta and angles}
    Let $f\colon \Delta_N \to \R^d$ be an almost $r$-embedding.
The value $\alpha^{(r)}(f)$ can be geometrically interpreted as follows.
    We have 
\[
	\Conf_r^{\Delta}(f)(x_1,...,x_r) = \left(f(x_1) - \tfrac{1}{r}\bigl(f(x_1) + \dots + f(x_r)\bigr), \dots , f(x_r) - \tfrac{1}{r}\bigl(f(x_1) + \dots + f(x_r)\bigr)\right),
\]
as depicted in Figure \ref{fig:r-tuple}.
Then, there are arbitrarily close $r$-tuples $x, y \in \Conf_r^{\Delta}(\Delta_N)$ with the property that the angle between the vectors $\Conf_r^{\Delta}(f)(x),~ \Conf_r^{\Delta}(f)(y) \in W_r^{\oplus d}$
is at least $\alpha^{(r)}(f)$.
In other words, the notion $\alpha^{(r)}(f)$ captures how large angles between values of $\Conf_r^{\Delta}(f)$ at two close points in the $r$-fold deleted product can be.
If $f$ is moreover continuous, then so is $\Conf_r^{\Delta}(f)$, and therefore $\alpha^{(r)}(f) = 0$.
So, in particular, having $\alpha^{(r)}(f)>0$ ensures that $f$ is discontinuous.

\begin{figure}[h]
    \centering
    \includegraphics[width=0.52\linewidth]{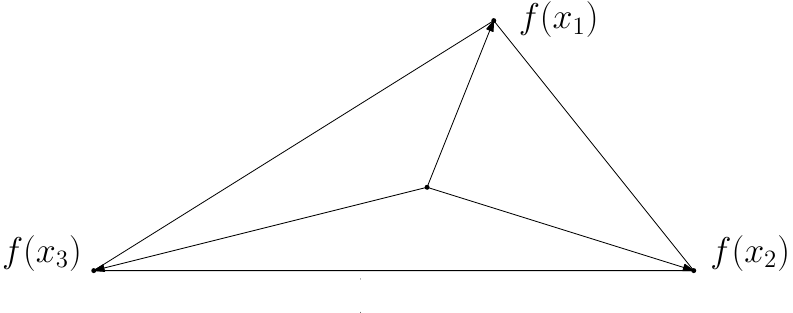}
    \caption{The three vectors in the triple $\Conf_3^{\Delta}(f)(x_1,x_2,x_3)$ for $r=3$ and $d=2$.}
    \label{fig:r-tuple}
\end{figure}
\end{remark}

We have the following technical lemma used to prove the topological Tverberg theorem.

\begin{proposition}[{\cite{BSSz1981,ozaydin1987}}] \label{proposition: nonexistence of equivariant map for tverberg}
    Let $d \ge 1$ and $r \ge 2$ be integers and $N = (r-1)(d+1)$.
Then, there exists an equivariant map
    \begin{equation*}
        \Conf_r^{\Delta}(\Delta_N) \longrightarrow_{\mathfrak{S}_r} S(W_r^{\oplus d})
    \end{equation*}
    if and only if $r$ is not a prime power.
\end{proposition}

We prove the following quantified version of Proposition \ref{proposition: nonexistence of equivariant map for tverberg}.

\begin{lemma} \label{lem: c>0 tverberg}
	Let $d \ge 1$ be an integer, let $r$ be a prime power, and let $N = (r-1)(d+1)$.
Then
	\[
		\inf \big\{\rho:~\exists \text{ continuous } \Conf_r^{\Delta}(\Delta_N) \longrightarrow_{\mathfrak{S}_r} \vrm{S(W_r^{\oplus d})}{\rho}\big\} ~=~ r_{d(r-1)-1}  \coloneqq \arccos(\tfrac{-1}{d(r-1)})
	\]
    is the diameter of a $d(r-1)$-dimensional regular simplex inscribed in the sphere $S(W_r^{\oplus d})$, equipped with the geodesic metric.
If $r$ is not a prime power, the infimum is zero.
\end{lemma}

\begin{proof}
    Throughout the proof we set $r_{d(r-1)-1} \coloneqq \arccos(\tfrac{-1}{d(r-1)})$ to ease the notation.
    Assume first that $r$ is a prime power and that there exists an equivariant map 
    \[
    	\Conf_r^{\Delta}(\Delta_N) ~\longrightarrow_{\mathfrak{S}_r}~ \vrm{S(W_r^{\oplus d})}{\rho}
    \]
    for some $0 < \rho < r_{d(r-1)-1}$.
From \cite[Proposition 5.3]{AAF} it follows that there exists an equivariant homotopy equivalence
    \[
    	\vrm{S(W_r^{\oplus d})}{\rho} ~\xrightarrow{~~\simeq~~}_{\mathfrak{S}_r}~ S(W_r^{\oplus d}),
    \]
    so the equivariant composition
    \[
    	\Conf_r^{\Delta}(\Delta_N) ~\longrightarrow_{\mathfrak{S}_r}~ \vrm{S(W_r^{\oplus d})}{\rho} ~\longrightarrow_{\mathfrak{S}_r}~ S(W_r^{\oplus d})
    \]
    is continuous, and hence contradicts Proposition~\ref{proposition: nonexistence of equivariant map for tverberg}.
Thus, the infimum from the statement of the lemma is at least $r_{d(r-1)-1}$.

On the other hand, let us show that an equivariant map
    \begin{equation*}
        \Conf_r^{\Delta}(\Delta_N) ~\longrightarrow_{\mathfrak{S}_r}~ \vrm{S(W_r^{\oplus d})}{r_{d(r-1)-1}}
    \end{equation*}
    exists.
The homotopy type of the codomain, obtained in \cite[Theorem 5.4]{AAF}, is
    \[
        \vrm{S(W_r^{\oplus d})}{r_{d(r-1)-1}} \simeq S(W_r^{\oplus d}) * \tfrac{\so((r-1)(d+1))}{A_{(r-1)(d+1)+1}}.
    \]
    In particular, it is $(r-1)d$-connected, so the equivariant map exists because there are no obstructions to defining it since the domain $\Conf_r^{\Delta}(\Delta_N)$ is of dimension $(r-1)d$.

Finally, let us assume that $r$ is not a prime power. Again by Proposition~\ref{proposition: nonexistence of equivariant map for tverberg} there exists an $\mathfrak{S}_r$-map $\Conf_r^{\Delta}(\Delta_N) \to S(W_r^{\oplus d}) = \vrm{S(W_r^{\oplus d})}{0}$, so the infimum from the statement is zero in this case.
\end{proof}

We may now prove the main theorem of this section.

\begin{proof}[Proof of Theorem \ref{theorem: discontinuous tverberg}]
	Assume $f$ is almost $r$-injective.
Then, as discussed above, the map
	\[
		\Conf_r^{\Delta}(\Delta_N) \xrightarrow{\Conf_r^{\Delta}(f)}_{\mathfrak{S}_r} W_r^{\oplus d} \setminus \{0\} \xrightarrow{~\nu~}_{\mathfrak{S}_r} S(W_r^{\oplus d})
	\]
	is well-defined.
Since $\Conf_r^{\Delta}(\Delta_N)$ is compact, a straightforward generalization of \cite[Lemma~7.4]{GH-BU-VR} to metric thickenings implies that for any $\varepsilon > 0$ there exists a parameter $\rho = \rho(\varepsilon)$ such that the induced function 
    \begin{equation*}
        \vrm{\Conf_r^{\Delta}(\Delta_N)}{\rho} \longrightarrow_{\mathfrak{S}_r} \vrm{S(W_r^{\oplus d})}{\alpha^{(r)}(f) + \varepsilon}
    \end{equation*} 
    is continuous and $\mathfrak{S}_r$-equivariant.
An $\mathfrak{S}_r$-equivariant inclusion $\Conf_r^{\Delta}(\Delta_N) \hookrightarrow \vrm{\Conf_r^{\Delta}(\Delta_N)}{\rho}$ and Lemma \ref{lem: c>0 tverberg} imply $\alpha^{(r)}(f) + \varepsilon ~\ge~ r_{d(r-1)-1}$, so the result follows by letting $\varepsilon \to 0$.
\end{proof}

\subsection{Relating $\alpha^{(r)}$ to the modulus of discontinuity and $r$-thickness}

We now relate $\alpha^{(r)}(f)$ to more concrete quantities $\delta(f)$ and $\kappa^{(r)}(f)$, where the latter is defined as follows:

\begin{definition}[{$r$-thickness}]
    Let $N\ge 0$, $d \ge 1$, and $r \ge 2$ be integers.
    For a function $f \colon \Delta_N \to \R^r$ we define its \emph{$r$-thickness} as the value
    \begin{equation*}
\kappa^{(r)}(f) := \inf_{(x_{1}, \ldots , x_{r}) \in \Conf_r^{\Delta}(\Delta_{N})} \frac{1}{r} \sqrt{\sum_{i = 1}^{r} \sum_{j = 1}^{r} \|f(x_{i}) - f(x_{j})\|^{2}}.
\end{equation*}
\end{definition}
    
The $r$-thickness $\kappa^{(r)}(f)$ measures how close $f$ is to violating almost $r$-injectivity. The quantity $\kappa^{(2)}$ represents the Gromov--Guth thickness~\cite[Section~2]{gromov2012generalizations}.

Specifically, we will show in Theorem~\ref{thm:quant-tv-r-thick} that $\sqrt{2} \cdot \sin(\alpha^{(r)}(f) / 2) \cdot \kappa^{(r)}(f) \le \delta(f)$.
Then lower bounds on $\alpha^{(r)}(f)$, e.g.\ Theorem~\ref{thm:quant-vk-f} and Theorem~\ref{theorem: discontinuous tverberg}, imply lower bounds on $\delta(f) / \kappa^{(r)}(f)$, which have nice concrete interpretations.
%For example, in the case of Tverberg and $r$ a prime power, we learn that any almost $r$-injective function $f \colon \Delta_{(r-1)(d+1)} \to \R^d$ is either highly discontinuous (i.e., $\delta(f) \ge  0$ is ``large'') or nearly violates $r$-injectivity ($\delta(f)$ being ``small'' implies that $\kappa^{(r)}(f)$ is ``small'').
For example, in the case of Tverberg and $r$ a prime power, we learn that any almost $r$-injective function $f \colon \Delta_{(r-1)(d+1)} \to \R^d$ has discontinuity that is lower bounded by a factor of its $r$-thickness.

We will need four lemmas.
First we consider the effect of normalization on distances in $\R^{n}$:

\begin{lemma}
\label{lem:kappa0}
Let $x, y \in \R^{n} \setminus \{0\}$.
Then
\begin{equation*}
\min(\|x\|, \|y\|) \cdot \left\|\frac{x}{\|x\|} - \frac{y}{\|y\|}\right\| \le \|x - y\|.
\end{equation*}
\end{lemma}

\begin{proof}
Assume without loss of generality that $\|x\| \le \|y\|$.
Since the desired inequality is preserved under replacing $x, y$ with $cx, cy$ for $c \ne 0$, we may assume $\|x\| = 1$.
By an orthogonal change of basis we may assume $x = (1, 0, 0, \ldots , 0)$, $y = (y_{1}, y_{2}, 0, \ldots , 0)$ with $\|y\| \ge 1$.
Then we may assume $n = 2$, and moreover we may replace $\R^{2}$ with $\C$, so that $x = 1$ and $y = rz$, where $r \ge 1$ and $z \in \C$ with $\|z\| = 1$.
Hence it suffices to prove $|1 - z| \le |1 - rz|$ for $r \ge 1$ and $z \in \C$.

Now consider the function $f \colon \R \to \C$ defined by $f(r) = |1 - rz|^{2}$, which takes values in $\R$.
We have $f(r) = (1 - rz)(1 - r\bar{z})$, so $f'(r) = 2r - z - \bar{z}$.
Since $\|z\| = \|\bar{z}\| = 1$, we have $f'(r) \ge 0$ for $r \ge 1$.
Therefore, $f(r) \ge f(1)$ for all $r \ge 1$, which gives the desired inequality.
\end{proof}

Next we consider the effect of normalization on the modulus of discontinuity $\delta$:

\begin{lemma}
\label{lem:kappa1}
Let $V$ be a real inner product space, and define $\nu : V \setminus \{0\} \to S(V)$ by $\nu(x) = x / \|x\|$.
Let $g \colon X \to V \setminus \{0\}$ be a function, and define $\kappa(g) = \inf_{x \in X}\|g(x)\|$.
Then
\begin{equation*}
2 \sin(\delta(\nu \circ g) / 2) \cdot \kappa(g) \le \delta(g).
\end{equation*}
\end{lemma}

\begin{proof}
Assume $V = \mathbb{R}^{n}$ by a change of basis, and then the result follows by applying Lemma~\ref{lem:kappa0} to pairs $g(x), g(y)$ with $x, y \in X$, using $\min(\|g(x)\|, \|g(y)\|) \ge \kappa(g)$.
Note that $\|\frac{g(x)}{\|g(x)\|} - \frac{g(y)}{\|g(y)\|}\|$ measures the Euclidean distance between $(\nu \circ g)(x), (\nu \circ g)(y)$, but $\delta(\nu \circ g)$ uses the geodesic distance in $S(V)$, so geodesic distances $d$ must be converted to Euclidean distances $2\sin(d / 2)$.
\end{proof}

Since $\Conf_{r}^{\Delta}(f)$ has $r$ components each defined using $f$, we can prove:

\begin{lemma}
\label{lem:kappa2}
For any function $f \colon \Delta_{N} \to \mathbb{R}^{d}$, we have $\delta(\Conf_{r}^{\Delta}(f)) \le \sqrt{r} \cdot \delta(f)$.
\end{lemma}

\begin{proof}
Let $f^{r} \colon \Delta_{N}^{\times r} \to (\mathbb{R}^{d})^{\oplus r}$ be defined by $f^{r}(x_{1}, \ldots , x_{r}) = (f(x_{1}), \ldots , f(x_{r}))$.
Since $\Conf_{r}^{\Delta}(f)$ is defined on a smaller domain by applying $f^{r}$ and then an orthogonal projection to $W_{r}^{\oplus d}$ which does not increase distances, we have $\delta(\Conf_{r}^{\Delta}(f)) \le \delta(f^{r})$, so it suffices to prove $\delta(f^{r}) \le \sqrt{r} \cdot \delta(f)$.

To prove this, let $\varepsilon > 0$; then for all $x \in \Delta_{N}$ there exists an open neighborhood $U_{x} \ni x$ such that $\diam(f(U_{x})) \le \delta(f) + \varepsilon$.
Now consider $(x_{1}, \ldots , x_{r}) \in \Delta_{N}^{\times r}$, and note that $\prod_{i} U_{x_{i}} \ni (x_{1}, \ldots , x_{r})$ is open in $\Delta_{N}^{\times r}$.
Then 
\[
    \diam\Big(f^{r}\big(\prod_{i} U_{x_{i}}\big)\Big)^{2} \le \sum_{i} \diam(f(U_{x_{i}}))^{2} \le r(\delta(f) + \varepsilon)^{2}.
\]
Taking $\varepsilon \to 0$ gives the result.
\end{proof}

Finally, we can relate $\kappa$ (as defined in Lemma~\ref{lem:kappa1}) to $\kappa^{(r)}$:

\begin{lemma}
\label{lem:kappa3}
For any function $f \colon \Delta_{N} \to \mathbb{R}^{d}$, we have $\sqrt{2} \cdot \kappa(\Conf_{r}^{\Delta}(f)) = \sqrt{r} \cdot \kappa^{(r)}(f)$.
\end{lemma}

\begin{proof}
Consider the projection $\pi \colon (\R^{d})^{\oplus r} \to W_{r}^{\oplus d}$ defined by $(y_{i})_{i=1}^r \mapsto (y_{i} - \frac{1}{r}\sum_{j=1}^ry_{j})_{i=1}^r$.
It is straightforward to verify by inner product expansion that
\[
    2r \sum_{i=1}^r\Big\| y_{i} - \frac{1}{r} \sum_{j=1}^r y_{j}\Big\|^{2} = \sum_{i, j=1}^r \|y_{i} - y_{j}\|^{2}.
\]
Then take $y_{i} = f(x_{i})$, and take the infimum over $(x_{1}, \ldots , x_{r}) \in \Conf_{r}^{\Delta}(\Delta_{N})$, which gives the identity $2r \cdot \kappa(\Conf_{r}^{\Delta}(f))^{2} = r^{2} \cdot \kappa^{(r)}(f)^{2}$, and the result follows.
\end{proof}

Now we are ready to relate $\alpha^{(r)}(f)$ to the quantities $\kappa^{(r)}(f)$ and $\delta(f)$:

\begin{proposition}
\label{prop:kappa}
For any $r$-injective function $f \colon \Delta_{N} \to \mathbb{R}^{d}$, we have
\[
    \delta(f) \ge \sqrt{2} \cdot \sin(\alpha^{(r)}(f) / 2) \cdot \kappa^{(r)}(f).
\]
\end{proposition}

\begin{proof}
By the definition of $\alpha^{(r)}$, we have $\alpha^{(r)}(f) = \delta(\nu \circ \Conf_{r}^{\Delta}(f))$, so by Lemma~\ref{lem:kappa1}, we have 
\[
    \delta(\Conf_{r}^{\Delta}(f)) \ge 2\sin(\alpha^{(r)}(f) / 2) \cdot \kappa(\Conf_{r}^{\Delta}(f)).
\]
The result follows from Lemmas~\ref{lem:kappa2} and~\ref{lem:kappa3}.
\end{proof}

For example, we obtain the following more concrete quantitative Tverberg result saying that for suitable dimensions, any almost $r$-injective function $f$ is either highly discontinuous or nearly violates $r$-injectivity:

\begin{theorem}
\label{thm:quant-tv-r-thick}
    Let $d \ge 1$ be an integer, $r$ a prime power, and $f \colon \Delta_{(r - 1)(d + 1)} \to \R^{d}$ any function.
    Then $f$ must satisfy $\delta(f) \ge \kappa^{(r)}(f) \cdot \sqrt{1 + \frac{1}{d(r - 1)}}$.
\end{theorem}

\begin{proof}
    If $f$ is not almost $r$-injective, then $\kappa^{(r)}(f) = 0$, and the inequality holds trivially.
    Otherwise, by Proposition~\ref{prop:kappa} and Theorem~\ref{theorem: discontinuous tverberg}, we have
    \[
        \delta(f) \ge \sqrt{2} \cdot \kappa^{(r)}(f) \cdot \sin\big(\arccos(\tfrac{-1}{d(r - 1)}) / 2\big).
    \]
    Simplifying using the half-angle formula for sine gives the result.
\end{proof}

\section*{Acknowledgments}

We are grateful to Evgeniya Lagoda and Sunhyuk Lim for useful discussions, and anonymous referees for helpful comments on the paper.

This version of the article has been accepted for publication, after peer review, but is not the Version of Record and does not reflect post-acceptance improvements, or any corrections. The Version of Record is
available online at: \url{https://doi.org/10.1007/s41468-026-00259-7}.

\bibliographystyle{plain}
\bibliography{GH-BU-VR.bib}

\begin{thebibliography}{10}

\bibitem{AA-VRS1}
Micha{\l} Adamaszek and Henry Adams.
\newblock The {V}ietoris--{R}ips complexes of a circle.
\newblock {\em Pacific J. Math.}, 290:1--40, 2017.

\bibitem{AAF}
Micha{\l} Adamaszek, Henry Adams, and Florian Frick.
\newblock Metric reconstruction via optimal transport.
\newblock {\em SIAM J. Appl. Algebra Geom.}, 2(4):597--619, 2018.

\bibitem{GH-BU-VR}
Henry Adams, Johnathan Bush, Nate Clause, Florian Frick, Mario G\'{o}mez,
  Michael Harrison, R.~Amzi Jeffs, Evgeniya Lagoda, Sunhyuk Lim, Facundo
  M\'{e}moli, Michael Moy, Nikola Sadovek, Matt Superdock, Daniel Vargas,
  Qingsong Wang, and Ling Zhou.
\newblock Gromov--{H}ausdorff distances, {B}orsuk--{U}lam theorems, and
  {V}ietoris--{R}ips complexes.
\newblock {\em arXiv preprint arXiv:2301.00246}, 2023.

\bibitem{ABF2}
Henry Adams, Johnathan Bush, and Florian Frick.
\newblock The topology of projective codes and the distribution of zeros of odd
  maps.
\newblock {\em Michigan Math. J.}, 74(4):775--796, 2024.

\bibitem{avvakumov2021eliminating}
Sergey Avvakumov, Isaac Mabillard, Arkadiy~B. Skopenkov, and Uli Wagner.
\newblock {Eliminating higher-multiplicity intersections. III. Codimension 2}.
\newblock {\em Israel J. Math.}, 245(2):501--534, 2021.

\bibitem{BSSz1981}
Imre B\'ar\'any, Senya~B. Shlosman, and Andr\'as Sz\H{u}cs.
\newblock On a topological generalization of a theorem of {T}verberg.
\newblock {\em J. Lond. Math. Soc.}, 23:158--164, 1981.

\bibitem{barzdin1993realization}
Ya~M. Barzdin.
\newblock On the realization of networks in three-dimensional space.
\newblock In {\em Selected works of A.\ N.\ Kolmogorov}, pages 194--202.
  Springer, 1993.

\bibitem{blagojevic2015barycenters}
Pavle~V.M. Blagojevi{\'c}, Florian Frick, and G{\"u}nter~M. Ziegler.
\newblock {Barycenters of polytope skeleta and counterexamples to the
  topological Tverberg conjecture, via constraints}.
\newblock {\em J. Eur. Math. Soc. (JEMS)}, 21(7):2107––2116, 2019.

\bibitem{borsuk1933drei}
Karol Borsuk.
\newblock {Drei S{\"a}tze {\"u}ber die $n$-dimensionale euklidische
  Sph{\"a}re}.
\newblock {\em Fundamenta Mathematicae}, 20(1):177--190, 1933.

\bibitem{dubins1981equidiscontinuity}
Lester Dubins and Gideon Schwarz.
\newblock Equidiscontinuity of {B}orsuk--{U}lam functions.
\newblock {\em Pacific J. Math.}, 95(1):51--59, 1981.

\bibitem{flores1933n}
Antonio Flores.
\newblock {{\"U}ber $n$-dimensionale Komplexe, die im $R^{2n+1}$ absolut
  selbstverschlungen sind}.
\newblock In {\em Ergeb. Math. Kolloq}, volume~34, pages 4--6, 1933.

\bibitem{freedman2014geometric}
Michael Freedman and Vyacheslav Krushkal.
\newblock Geometric complexity of embeddings in $\mathbb{R}^d$.
\newblock {\em Geom. Funct. Anal.}, 24(5):1406--1430, 2014.

\bibitem{frick2015counterexamples}
Florian Frick.
\newblock {Counterexamples to the topological Tverberg conjecture}.
\newblock {\em Oberwolfach Rep.}, 12(1):318, 2015.

\bibitem{gillespie2024vietoris}
Patrick Gillespie.
\newblock Vietoris thickenings and complexes are weakly homotopy equivalent.
\newblock {\em J. Appl. Comput. Topol.}, 8(1):35--53, 2024.

\bibitem{gromov2012generalizations}
Misha Gromov and Larry Guth.
\newblock {Generalizations of the Kolmogorov--Barzdin embedding estimates}.
\newblock {\em Duke Math. J.}, 161(13):2549--2603, 2012.

\bibitem{Haefliger}
Andr{\'e} Haefliger.
\newblock Plongements diff{\'e}rentiables dans le domaine stable.
\newblock {\em Comment. Math. Helv.}, 37:155--176, 1962.

\bibitem{tverberg1966}
Tverberg Helge.
\newblock A generalization of {R}adon’s theorem.
\newblock {\em J. Lond. Math. Soc.}, 41:123–128, 1966.

\bibitem{james1963immersion}
Ioan~M. James.
\newblock On the immersion problem for real projective spaces.
\newblock {\em Bull. Amer. Math. Soc.}, 69(2):231--238, 1963.

\bibitem{kuratowski1930probleme}
Casimir Kuratowski.
\newblock Sur le probleme des courbes gauches en topologie.
\newblock {\em Fundamenta Math.}, 15(1):271--283, 1930.

\bibitem{lim2025g}
Sunhyuk Lim and Facundo Memoli.
\newblock {The $G$-Gromov-Hausdorff Distance and Equivariant Topology}.
\newblock {\em arXiv preprint arXiv:2506.15414}, 2025.

\bibitem{lim2021gromov}
Sunhyuk Lim, Facundo M{\'e}moli, and Zane Smith.
\newblock The {G}romov--{H}ausdorff distance between spheres.
\newblock {\em Geom. Topol.}, 27(9):3733--3800, 2023.

\bibitem{mabillard2014eliminating}
Isaac Mabillard and Uli Wagner.
\newblock {Eliminating Tverberg points, I. An analogue of the Whitney trick}.
\newblock In {\em Proceedings of the thirtieth annual symposium on
  Computational geometry}, pages 171--180, 2014.

\bibitem{matousek2003using}
Ji\v{r}\'{i} Matou\v{s}ek.
\newblock {\em Using the Borsuk--{U}lam theorem}.
\newblock Springer Science \& Business Media, 2008.

\bibitem{milnor1974characteristic}
John~W. Milnor and James~D. Stasheff.
\newblock {\em Characteristic classes}.
\newblock Number~76. Princeton University Press, 1974.

\bibitem{moyVRmS1}
Michael Moy.
\newblock Vietoris--{R}ips metric thickenings of the circle.
\newblock {\em J. Appl. Comput. Topol.}, 7(4):831--877, 2023.

\bibitem{sanderson1964immersions}
Brian~J. Sanderson.
\newblock Immersions and embeddings of projective spaces.
\newblock {\em Proc. Lond. Math. Soc.}, 3(1):137--153, 1964.

\bibitem{tom1987transgroups}
Tammo tom Dieck.
\newblock {\em Transformation groups}, volume~8 of {\em De Gruyter Studies in
  Mathematics}.
\newblock Walter de Gruyter \& Co., Berlin, 1987.

\bibitem{van1933komplexe}
Egbert~R. Van~Kampen.
\newblock {Komplexe in euklidischen R{\"a}umen}.
\newblock In {\em Abh. Math. Seminar Univ. Hamburg}, volume~9, pages 72--78.
  Springer, 1933.

\bibitem{wagner1937eigenschaft}
Klaus Wagner.
\newblock {\"U}ber eine {E}igenschaft der ebenen {K}omplexe.
\newblock {\em Math. Ann.}, 114(1):570--590, 1937.

\bibitem{Weber}
Claude Weber.
\newblock Plongements de polyedres dans le domaine metastable.
\newblock {\em Comment. Math. Helv.}, 42(1):1--27, 1967.

\bibitem{whitney44}
Hassler Whitney.
\newblock The self-intersections of a smooth {$n$}-manifold in {$2n$}-space.
\newblock {\em Ann. of Math. (2)}, 45:220--246, 1944.

\bibitem{ozaydin1987}
Murad Özaydin.
\newblock Equivariant maps for the symmetric group.
\newblock {\em Preprint, \url{http://digital.library.wisc.edu/1793/63829}},
  1987.

\end{thebibliography}

\end{document}